\documentclass[12pt,reqno]{amsart}

\usepackage[initials]{amsrefs}
\usepackage{amssymb,amsmath,amsfonts,latexsym,amsthm} 

\usepackage[a4paper,nomarginpar]{geometry}

\usepackage[english]{babel}

\usepackage[latin1]{inputenc}

\newtheorem{theorem}{Theorem}[section]

\newtheorem{proposition}[theorem]{Proposition}
\newtheorem{lemma}[theorem]{Lemma}

\theoremstyle{definition}

\newtheorem{remarks}[theorem]{Remarks}
\newtheorem{definition}[theorem]{Definition}

\newcommand{\N}{\mathbb{N}} 
\newcommand{\R}{\mathbb{R}} 
\newcommand{\C}{\mathbb{C}} 
\newcommand{\D}{\mathbb{D}} 
\newcommand{\T}{\mathbb{T}} 
\newcommand{\ovl}{\overline}
\newcommand{\al}{\alpha}
\newcommand{\ve}{\varepsilon}

\newcommand{\dis}{\displaystyle}

\begin{document}

\title[Subspaces of compositionally frequently hypercyclic functions]
{Subspaces of frequently hypercyclic functions for sequences of composition operators}

\date{}

\author[Bernal]{L.~Bernal-Gonz\'alez}
\address{Departamento de An\'{a}lisis Matem\'{a}tico, Facultad de Matem\'{a}ticas
\newline\indent Instituto de Matem\'aticas Antonio de Castro Brzezicki (IMUS)
\newline\indent Universidad de Sevilla
\newline\indent Avda. Reina Mercedes s/n
\newline\indent 41080 Sevilla, Spain.}
\email{lbernal@us.es}

\author[Calder\'on]{M.C.~Calder\'on-Moreno}
\address{Departamento de An\'{a}lisis Matem\'{a}tico, Facultad de Matem\'{a}ticas
\newline\indent Instituto de Matem\'aticas Antonio de Castro Brzezicki (IMUS)
\newline\indent Universidad de Sevilla
\newline\indent Avda. Reina Mercedes s/n
\newline\indent 41080 Sevilla, Spain.}
\email{mccm@us.es}

\author[Jung]{A.~Jung}
\address{Fachbereich IV Mathematik
\newline\indent Universit\"at Trier
\newline\indent D-54286 Trier, Germany}
\email{andreas.tibor.jung@gmail.com}

\author[Prado]{J.A.~Prado-Bassas}
\address{Departamento de An\'{a}lisis Matem\'{a}tico
\newline\indent Facultad de Matem\'{a}ticas
\newline\indent Universidad de Sevilla
\newline\indent Avda. Reina Mercedes s/n
\newline\indent 41080 Sevilla, Spain.}
\email{bassas@us.es}


\thanks{}

\begin{abstract}
In this paper, a criterion for a sequence of composition operators defined on the space of holomorphic functions in a complex domain to be frequently hypercyclic is provided. Such a criterion improves some already known special cases and, in addition, it is also valid to provide dense vector subspaces as well as large closed ones consisting entirely, except for zero, of functions that are frequently hypercyclic.
\end{abstract}

\maketitle

\section{Introduction, notation and background}

\quad The phenomenon of hypercyclicity has become a trend in the last three decades. Roughly speaking, it means
density of some orbit of a vector under the action of an operator or a sequence of operators. When this density
is quantified in some optimal sense, the property of frequent hypercyclicity
-- a concept coined by Bayart and Grivaux \cite{bayartgrivaux2006} -- arises naturally.
In this paper we are concerned with the study of frequent hypercyclicity of sequences of composition operators
acting on holomorphic functions defined in a domain of the complex plane.
But prior to going on, let us fix some notation and definitions, mostly standard.

\vskip 4pt

By \,$\D$, $\ovl{\D}$ and $\T$ \,we denote, respectively, the open unit disc $\{|z| < 1\}$, the closed
unit disc $\{|z| \le 1\}$ and the unit circle $\{|z| = 1\}$ in the complex plane \,$\C$.
A {\it domain} is a nonempty connected open subset $G \subset \C$.
Its one-point compactification \,$G \cup \{\infty \}$ \,will be represented by \,$G_\infty$.
A domain \,$G \subset \C$ \,is called {\it simply connected} provided that \,$\C_\infty \setminus G$ \,is connected.
The symbol \,$H(G)$ \,will stand for the space of holomorphic functions on $G$.
It becomes an F-space --that is, a complete metrizable topological vector space-- under
the topology of uniform convergence on compact subsets of $G$.

\vskip 4pt

Next, we recall the notion of hypercyclicity.
For a good account of concepts, results and history concerning
this topic, the reader is referred to the books \cite{bayartmatheron2009,grosseperis2011}.
Let $X$ and $Y$ be two (Hausdorff) topological
vector spaces. A sequence
\,$T_n : X \to Y$ $(n \in \N )$ of continuous linear mappings is said to be {\it hypercyclic} provided that there is a vector
$x_0 \in X$ (called hypercyclic for $(T_n)_n$) such that the orbit \,$\{T_n x_0: \, n \in \N\}$
\,is dense in \,$Y$. An operator $T$ on $X$ (i.e., a continuous linear selfmapping $X \to X$) is said to be {\it hypercyclic} if the sequence $(T^{n})_n$ of
its iterates is hypercyclic. The corresponding sets of hypercyclic vectors will be respectively denoted by
$HC((T_n)_n)$ and $HC(T)$.
As a stronger notion of hypercyclicity, the property of frequent hypercyclicity  was introduced in \cite{bayartgrivaux2006}
for operators, and then extended in \cite{bonillagrosse2007} to sequences of linear mappings.
Given a subset $A\subset\N$, the lower density of $A$ is defined as
$\underline{\rm dens}\,(A) = \liminf_{n\to\infty} \frac{{\rm card} \, (A \cap \{1,\dots,n\})}{n}$.
Assume that $X$ and $Y$ are topological vector spaces. Then a sequence $T_n : X \to Y$ $(n \in \N )$ of continuous linear mappings is
called {\it frequently hypercyclic} provided that there is a vector
$x_0 \in X$ (called frequently hypercyclic for $(T_n)_n$) such that
$\underline{\rm dens} \,(\{ n \in \N : \, T_n x_0 \in V \}) > 0$
for every nonempty open set $V \subset Y$. An operator $T$ on $X$ is said to be {\it frequently hypercyclic} if the sequence $(T^{n})_n$ of
its iterates is frequently hypercyclic. The corresponding sets of frequently hypercyclic vectors will be respectively denoted by
$FHC((T_n)_n)$ and $FHC(T)$. Each set $HC(T),\, FHC(T)$ is dense as soon as it is nonempty.

\vskip 4pt

Our main concern enters the realm of composition operators. If \,$G \subset \C$ \,is a domain,
then every holomorphic selfmap \,$\varphi : G \to G$ \,defines a {\it composition operator} on
\,$H(G)$ given by
\,$C_\varphi : f \in H(G) \longmapsto f \circ \varphi \in H(G)$.
By $H(G,G)$, $H_{1-1}(G)$ \,and \,${\rm Aut} (G)$ \,we denote, respectively, the family of such selfmaps, the
subfamily of all univalent (i.e., one-to-one) members of $H(G,G)$, and the subfamily of all automorphisms of $G$ (that is, the bijective
holomorphic functions $G \to G$). Hence \,${\rm Aut} (G) \subset H_{1-1}(G) \subset H(G,G)$.
According to \cite{bernalmontes1995}, a sequence \,$(\varphi_n)_n \subset H(G,G)$ \,is said to be {\it runaway}
\,whenever it satisfies the following property: for every compact set \,$K \subset G$, there exists \,$N \in \N$ \,such that
$K \cap \varphi_N (K) = \varnothing$.
A function \,$\varphi \in H(G,G)$ \,is said to be {\it runaway} \,if the sequence \,$(\varphi^{n})$
\,of its compositional iterates 
is runaway.

\vskip 2pt

The dynamics of a sequence \,$(C_{\varphi_n})_n$ \,has been studied by several authors, so as to obtain, among others, the following
assertions, in which \,$G$ \,is assumed to be a simply connected domain and \,$\varphi \in H(G,G)$:
\begin{enumerate}
\item[$\bullet$] (\cite{bernalmontes1995}) Let $(\varphi_n)_{n} := (z \mapsto a_n z + b_n)_{n } \subset {\rm Aut} (\C )$ $(a_n,b_n \in \C ; \, a_n \ne 0)$
                 \,and \,$(\psi_n)_{n } := \big( z \mapsto k_n {z -a_n \over 1 - \ovl{a_n}z} \big)_{n} \subset {\rm Aut} (\D )$ $(|a_n| < 1 = |k_n|)$. Then \,$(C_{\varphi_n})_n$ \,is hypercyclic if and only if \,$\limsup_{n \to \infty} \min \{ |b_n|,|b_n / a_n| \} = +\infty$
                 (Birkhoff's theorem \cite{birkhoff1929} is the special case \,$a_n = 0$, $b_n = an$, with $a \ne 0$), and
                 $(C_{\psi_n})_n$ is hypercyclic if and only if \,$\limsup_{n \to \infty} |a_n| = 1$.
\item[$\bullet$] The operator \,$C_\varphi$ \,is hypercyclic if and only if $\varphi \in H_{1-1}(G)$
                 \,and has no fixed point in $G$ (\cite{shapiro1993,shapiro2001}), and if and only if $\varphi \in H_{1-1}(G)$ and is runaway (\cite{grossemortini2009}). Moreover, if \,$(\varphi_n)_n \subset H_{1-1}(G)$, then \,$(C_{\varphi_n})_n$ \,is hypercyclic if and only if \,$(\varphi_n)_n$ \,is runaway (\cite{grossemortini2009}).
\item[$\bullet$] (\cite{bayartgrivaux2004,bayartgrivaux2006}) Each translation $\tau_a : f \in H(\C ) \mapsto f( \cdot \, + a) \in H(\C )$
$(a \in \C \setminus \{0\})$ is frequently hypercyclic. If \,$\varphi \in {\rm Aut} (\D )$ has no fixed point in $\D$,
then $C_\varphi : H(\D ) \to H(\D )$ is also frequently hypercyclic.
\item[$\bullet$] (\cite{bes2013}) The operator \,$C_\varphi$ \,is frequently hypercyclic if and only if it is hypercyclic, so if and only
                 if \,$\varphi \in H_{1-1}(G)$ \,and has no fixed point in $G$.
\item[$\bullet$] (\cite{bernal2013}) Let $(\varphi_n)_{n } := (z \mapsto a_n z + b_n)_{n } \subset {\rm Aut} (\C )$ $(a_n,b_n \in \C ; \, a_n \ne 0)$.
Assume that there is an unbounded nondecreasing sequence $(\omega_n )_{n } \subset (0,+\infty )$ satisfying
\,$\lim_{n \to \infty} (|b_n| - \omega_n |a_n|) = +\infty$ \,and
\,$|b_m - b_n| \ge \omega_{m-n} (|a_m| + |a_n|)$
\,for all $m,n \in \N$ with $m > n$.
Then \,$(C_{\varphi_n})_n$ \,is frequently hypercyclic on $H(\C )$.
Moreover, if $(b_n) \subset \C$ is a sequence with
$\lim_{k \to \infty} \inf_{n \in \N} |b_{n+k} - b_n| = +\infty$,
then the sequence of translations $(\tau_{b_n})$ is frequently hypercyclic on $H(\C )$.
\end{enumerate}

Recently, a new terminology has been coined with the aim of suggesting largeness in an algebraic sense.
This terminology enters the new subject of {\it lineability,} for whose background we refer the reader to the
survey \cite{bernalpellegrinoseoane2014} and the book \cite{aronbernalpellegrinoseoane}.
Assume that $X$ is a vector space and that $A \subset X$. The set $A$ is said to be {\it lineable} in $X$ provided that
there is an infinite dimensional vector space $M$ such that $M \subset A \cup \{0\}$. If $X$ is, in addition, a
topological vector space, then $A$ is called: {\it dense-lineable} in $X$  if there exists a dense
vector subspace $M \subset X$ such that \,$M \subset A \,\cup \,\{0\}$; {\it maximal dense-lineable} in $X$ if, further,
dim$(M)$ = dim$(X)$; {\it spaceable} in $X$ if there exists a closed infinite dimensional vector subspace
$M \subset X$ such that $M \subset A \cup \{0\}$.

\vskip 4pt

Let $X,Y$ be topological vector spaces, $T$ an operator on $X$ and $T_n : X \to Y$ $(n \in \N )$ be a sequence of continuous linear mappings.
By using the preceding language, it is known (see \cite{bes1999,bourdon1993,herrero1991,wengenroth2003}) that if $T$ is hypercyclic, then \,$HC(T)$ \,is always dense-lineable in $X$, even maximal dense-lineable if
\,$X$ \,is a Banach space
(in both cases, the existing dense subspace $M$ is $T$-invariant).
Sufficient conditions for $HC((T_n)_n)$ to be lineable or dense-lineable were furnished in \cite[Theorems 1--2]{bernal1999b}
(for maximal dense lineability of $HC((T_n)_n)$, see \cite{bernalordonez2014}).
In contrast to the property of dense-lineability, not every set \,$HC(T)$ \,is spaceable if \,$T$ is hypercyclic,
as Montes has shown in \cite{montes1996}.
Sufficient criteria for $HC(T)$ or $HC((T_n)_n)$ to be spaceable can be found in
\cite[Chapter 10]{grosseperis2011}, \cite[Section 4.5]{aronbernalpellegrinoseoane}
and the references contained in them. In particular, if \,$G \subset \C$ is a domain that is not
conformally equivalent to $\C \setminus \{0\}$ and $(\varphi_n)_n \subset {\rm Aut} (G)$ is a runaway sequence,
then $HC((C_{\varphi_n})_n)$ is spaceable in $H(G)$ \cite{bernalmontes1995a}; consequently, $HC(C_\varphi )$ is spaceable if
$\varphi \in {\rm Aut} (G)$ is runaway.
Assuming that \,$G$ \,is simply connected and \,$(\varphi_n)_n \subset H_{1-1}(G)$ \,is a runaway sequence, it can be extracted from
\cite[Theorem 3.2 and its proof]{grossemortini2009} together with \cite[Theorem 2]{bernal1999b} that
\,$HC((C_{\varphi_n})_n)$ is dense-lineable in $H(G)$.

\vskip 4pt

Concerning frequent hypercyclicity, Bayart and Grivaux \cite{bayartgrivaux2006} proved that if
$X$ is a separable F-space and
$T$ is a frequently hypercyclic operator on $X$, then \,$FHC(T)$ \,is dense-lineable
(again with $T$-invariance of the corresponding dense subspace). For conditions
for $FHC((T_n)_n)$ to be dense-lineable, see \cite{bernalordonez2014}.
As for the existence of large closed subspaces,
in \cite[Theorem 3]{bonillagrosse2012} a sufficient criterion for the spaceability of \,$FHC(T)$ \,is
given -- from which it is derived as an example
that $FHC(\tau_a)$ is spaceable in $H(\C )$ for each $a \in \C \setminus \{0\}$ -- while
Menet \cite[Theorem 2.12]{menet2015} provided a criterion to discover spaceability of \,$FHC((T_n)_n)$.
B\`es \cite{bes2013} proved in 2013 the spaceability in \,$H(G)$ \,of \,$FHC(C_\varphi )$ \,if it is assumed that \,$G \subset \C$
is a simply connected domain, $\varphi \in H_{1-1}(G)$ \,and \,$\varphi$ \,has no fixed point in $G$.

\vskip 4pt

As far as we know, there is not any general result about lineability of the fa\-mi\-ly of frequently hypercyclic functions
with respect to {\it sequences} \,of composition operators. The aim of this paper is to contribute to fill in this gap.
Specifically, we furnish sufficient conditions for a {\it sequence of composition operators} defined
on a given planar simply connected domain to be {\it frequently hypercyclic} and to enjoy the property that the family of its
corresponding frequently hypercyclic functions is not only nonempty but also it {\it contains vector spaces}
that are, in several senses, large. This will
be performed in Sections 2 and 3. In the final Section 4 a number of examples will be provided.

\section{Frequently hypercyclic sequences of composition operators}

\quad 
A sequence $(K_n)_{n }$ of compact subsets of a domain $G \subset \C$ is said to be
{\it exhaustive} \,if \,$G = \bigcup_{n = 1}^\infty K_n$ \,and \,$K_n \subset K_{n+1}^\circ$ \,for all \,$n \in \N$, where $A^\circ$ denotes the interior of $A$.
In particular, it satisfies that, given a compact set \,$K \subset G$, there exists \,$N \in \N$ \,such that \,$K \subset K_N$.
A compact set \,$K \subset \C$ \,is said to be \,{\it Mergelyan} \,whenever it lacks holes, that is, whenever
\,$\C \setminus K$ \,is connected. The symbol \,$\mathcal M (G)$ \,will stand for the family of Mergelyan compact subsets of \,$G$.
If $G$ is simply connected, an exhaustive sequence of compact subsets of \,$G$ \,satisfying \,$(K_n)_{n } \subset \mathcal M (G)$
\,always exists: see, e.g., \cite[Chap.~13]{rudin1987}.


\vskip 4pt

Given a compact subset $K\subset G$, $r > 0$ and a function $f \in H(G)$, by $\|f\|_K$ we mean the maximum of $|f(z)|$ over $K$, while
\,$B_K(f,r)$ \,will denote the set of all functions \,$h \in H(G)$ \,such that \,$\|h-f\|_K < r$.
The sets \,\,$B_K(f,r)$ \,form a base for the open sets of \,$H(G)$.

\vskip 4pt

We first give an easy necessary condition for a sequence \,$(\varphi_n)_n \subset H(G,G)$ \,to generate
a frequently hypercyclic sequence \,$(C_{\varphi_n})_n$ \,of composition operators.

\begin{proposition} \label{Prop: necessary condition}
If \,$(C_{\varphi_n})_n$ is frequently hypercyclic on \,$H(G)$, then \,$(\varphi_n)_n$ is {\em weakly frequently
runaway,} that is, for every compact set \,$K \subset G$, one has
$$
\underline{\rm dens}( \{n \in \N : \, K\cap \varphi_n (K) = \varnothing \} ) > 0.
$$
\end{proposition}

\begin{proof} Fix any compact set $K\subset G$. Let \,$n \in \N$ \,be a natural number such that \,$K\cap \varphi_n(K) \ne \varnothing$ \,and
pick \,$z_n \in K$ \,such that \,$\varphi_n (z_n) \in K$. Let \,$f\in FHC((C_{\varphi_n})_n)$ \,and
define \,$g$ \,as the constant function \,$g(z) :=  1+\|f\|_K$. Then
$$
\|g-C_{\varphi_n}f\|_K \ge |g(z_n)-f(\varphi_n(z_n))| \ge 1 + \|f\|_K - |f(\varphi_n(z_n))| \ge 1,
$$
where the facts \,$z_n \in K$ \,and \,$\varphi_n(z_n) \in K$ \,have been used. Hence
$$
\{ n \in \N: \, C_{\varphi_n} f \in B_K(g,1) \} \subset \{n \in \N : \, K\cap \varphi_n (K)  = \varnothing \}.
$$
But the frequent hypercyclicity of \,$f$ \,implies that the smaller set has positive lower density, so the bigger one too.
\end{proof}

The last proposition can suggest conditions guaranteeing that $FHC((C_{\varphi_n})_n) \ne \varnothing$. We restrict ourselves to the rather illustrative case
of simply connected domains. Neither the results nor the approaches in \cite{bes2013,bonillagrosse2012,menet2015}
will be used in our proofs.

\vskip 3pt

Prior to establish our main result in this section
(Theorem \ref{existence}, whose proof is inspired by the one of Theorem 4.1 in \cite{bernal2013}),
two auxiliary results are needed.

\begin{lemma} \label{Lemma splitting}
Any sequence of natural numbers with positive lower density can always be split
into infinitely many disjoint subsequences each of which has also positive lower density.
\end{lemma}

\begin{proof}
Let $A = \{n_1 < n_2 < \cdots < n_k < \cdots \}$ be a sequence in $\N$ with $\underline{\rm dens} (A) > 0$. This means that there is $C > 0$ with $n_k \le Ck$ for all $k \in \N$. If we define $A_1 = \{m_k\}_{k \ge 1} := \{n_1 < n_3 < n_5 < \cdots \}$, then $m_k \le C(2k-1) \le 2Ck$ for all $k \in \N$, so that $\underline{\rm dens} (A_1) > 0$. Now, we split $A \setminus A_1$ into $A_2 = \{p_k\}_{k \ge 1}:= \{n_2 < n_6 < n_{10} < n_{14} < \cdots\}$ \,and
\,$A \setminus (A_1 \cup A_2) = \{n_{4k}\}_{k \ge 1}$. Then $p_k \le C(4k - 2) \le 4Ck$ for all $k \ge 1$, which yields \,$\underline{\rm dens} (A_2) > 0$. Then we divide
$\{n_{4k}\}_{k \ge 1}$ into two sequences with positive lower density. It is clear that this procedure gives pairwise disjoint sets $A_n \subset \N$ with
\,$\underline{\rm dens} (A_n) > 0$ for all $n \in \N$.
\end{proof}

Our second auxiliary assertion (Theorem \ref{Nersesjan} below) is known as \textit{Nersesjan's approximation theorem} and
can be found in \cite[Chapter 4]{gaier1987} (see also \cite{nersesjan1971}).

\begin{definition}
Let \,$G\subset \C$ \,be a domain.
A relatively closed set \,$F \subset G$ \,is said to be a {\it Carleman set of \,$G$} \,provided that the following conditions hold:
\begin{enumerate}
\item[\rm (i)] $G_\infty \setminus F$ \,is connected,
\item[\rm (ii)] $G_\infty \setminus F$ \,is locally connected at \,$\infty$,
\item[\rm (iii)] $F$ ``lacks long islands'', that is, for every compact set \,$K \subset G$ \,there is a
neighborhood \,$V$ \,of \,$\infty$ \,in \,$G_\infty$
\,such that no component of \,$F^\circ$ 
meets both \,$K$ \,and \,$V$.
\end{enumerate}
\end{definition}

Note that (ii) in the last definition is equivalent to the following (see \cite[p.~143]{gaier1987}): for each neighborhood \,$U$ of \,$\infty$ \,there exists a
neighborhood \,$V \subset U$ of \,$\infty$ such that each point $z_0 \in V \setminus (F \cup \{\infty \})$ can be connected in \,$U \setminus F$ with a point that is arbitrarily close to \,$\infty$.

\begin{theorem} \label{Nersesjan}
A relatively closed set \,$F \subset G$ \,is a Carleman set of \,$G$ \,if and only if for every function $f$
continuous on $F$ and holomorphic on $F^\circ$ and every positive and continuous
function \,$\varepsilon : F \to [0,+\infty)$, there is a function \,$g\in H(G)$ \,such that
$$|f(z)-g(z)|<\varepsilon(z) \hbox{ \ for all \ } z \in F.$$
\end{theorem}

We say that a family \,$\mathcal F$ \,of subsets of a given set \,$X$ \,is {\it pairwise disjoint}
\,if \,$A \cap B = \varnothing$ \,for every pair of distinct \,$A,B \in \mathcal F$.
If \,$A \subset \C$ \,then \,$\partial A$ \,and \,$\partial_\infty A$ \,will stand for the boundary of \,$A$ \,in \,$\C$
\,and in \,$\C_\infty$, respectively, so that \,$\partial_\infty A$ \,equals either \,$\partial A$ \,or \,$\{\infty \} \cup \partial A$, depending
on whether \,$A$ \,is bounded or not.

\vskip 3pt

In order to formulate appropriately our criterion for frequent hypercyclicity, let us introduce the next concept.

\begin{definition}\label{Def: strongly freq run}
Let \,$G \subset \C$ \,be a domain and \,$(\varphi_n) \subset H(G,G)$. We say that $(\varphi_n)$ is {\em strongly frequently runaway} provided that there exists an exhaustive sequence \,$(K_\nu)_{\nu } \subset \mathcal M (G)$ \,as well as a family \,$\{A(\nu ):\,\nu \in\N\}$ \,of subsets of \ $\N$ \,satisfying the following conditions:
\begin{enumerate}
\item[\rm (P1)] $\underline{\rm dens} \, (A(\nu )) > 0$ \,for all \,$\nu \in \N$.
\item[\rm (P2)] Each of the families \,$\{A(\nu );\,\nu \in\N\}$ \,and \,${\mathcal K} := \{\varphi_n (K_\nu): \, n \in A(\nu ), \, \nu \in \N\}$ \,is pairwise disjoint.
\item[\rm (P3)] Given a compact subset \,$K \subset G$ \,there are only finitely many \,$L \in \mathcal K$ \,with \,$K \cap L \ne \varnothing$.
\end{enumerate}
\end{definition}

Note that thanks to (P3) the notion introduced in the preceding definition is stronger than the one of a weakly frequently runaway sequence given in Proposition \ref{Prop: necessary condition}.

\begin{theorem} \label{existence}
Let \,$G \subset \C$ \,be a simply connected domain and \,$(\varphi_n)_n$ be a strongly frequently runaway sequence in \,$H_{1-1}(G)$.
Then the sequence \,$(C_{\varphi_n})_n$ \,is frequently hypercyclic. In other words, the set $FHC((C_{\varphi_n})_n)$ is not empty.
\end{theorem}

\begin{proof}
Let $(P_l)_l$ be a dense sequence of $H(G)$, for instance the sequence of polynomials whose coefficients
have rational real and imaginary parts. Divide each set \,$A(\nu)$ \,into infinitely many disjoint subsets
\,$A(\nu,l)$ $(l \in \N )$, each of them with positive lower density. This is possible thanks to (P1) and Lemma \ref{Lemma splitting}.

\vskip 4pt

Now, we use the fact that if \,$K$ \,is a compact set with \,$N$ \,holes and \,$\varphi$ \,is an injective holomorphic mapping
on a neighborhood of \,$K$, then \,$\varphi (K)$ \,also has \,$N$ \,holes (see \cite[p.~276]{remmert1991}).
In particular, if \,$K \in \mathcal M (G)$, then \,$\varphi (K)$ \,lack holes. Define the set
$$
F := \bigcup_{\nu\in\N} \bigcup_{n \in A(\nu)} \varphi_n (K_\nu).
$$
From (P2), the sets \,$\varphi_n(K_\nu)$ \,are pairwise disjoint compact subsets of $G$ having no holes.
From (P3), it follows that they escape to the boundary of \,$G$. Therefore \,$F$ \,is closed in \,$G$, and \,$G_\infty \setminus F$ \,is connected
as well as locally connected at \,$\infty$. The components of \,$F^\circ$ \,are the sets \,$(\varphi_n (K_\nu))^\circ$.
Given a compact subset \,$K \subset G$ \,there are by (P3) only finitely many \,$L_1, \dots, L_N \in \mathcal K$ with
\,$K \cap L_j \ne \varnothing$ $(j = 1, \dots ,N)$. Then \,$V := \{\infty\} \cup (G \setminus (K \cup L_1 \cup \cdots \cup L_N))$
\,is a neighborhood \,of \,$\infty$ \,in \,$G_\infty$ \,satisfying that no component of \,$F^\circ$ \,meets both \,$K$ \,and \,$V$.
Hence \,$F$ \,is a Carleman set of \,$G$.
Now define \,$g : F\to \C$ \,as follows:
$$
g(z) := P_l(\varphi_n^{-1}(z)) \,\, \text{ if } \,z \in \varphi_n(K_\nu),\, n\in A(\nu,l) \, \text{ and } \, l,\nu \in \N .
$$
It is obvious that \,$g$ \,is continuous on \,$F$ \,and holomorphic in \,$F^\circ$. Hence, by Nersesjan's Theorem
(Theorem \ref{Nersesjan}), there exists \,$f\in H(G)$ \,such that \,$|f(z)-g(z)| < \varepsilon(z)$ \,for \,$z\in F$, where
$\varepsilon : G \to [0,+\infty)$ \,is a function that goes to zero as \,$z$ \,approaches the boundary of \,$G$
(for instance, take \,$\varepsilon(z) := \chi (z,\partial_\infty G)$, where \,$\chi$ \,denotes the chordal distance on \,$\C_\infty$).

\vskip 4pt

We claim that \,$f$ \,is frequently hypercyclic for \,$(C_{\varphi_n})_n$. Indeed, fix \,$l,\nu\in\N$ \,and \,$z\in K_\nu$.
Then, for \,$n\in A(l,\nu)$ \,we have
\begin{equation}\label{eq1}
|C_{\varphi_n}(f)(z) - P_l(z)| = |f(\varphi_n(z)) - g(\varphi_n(z))| < \varepsilon(\varphi_n(z)).
\end{equation}
But (P3) implies that \,$\varphi_n (z)$ \,tends to the boundary as \,$n\to\infty$
\,uniformly on \,$K_\nu$, so \,$\sup_{z \in K_\nu} \varepsilon (\varphi_n (z)) \to 0$. Hence, given \,$\delta > 0$,
there is \,$n_0 = n_0 (\nu ) \in\N$ \,such that \,$\sup_{z \in K_\nu} \varepsilon (\varphi_n (z)) < \delta$ \,for all \,$n \ge n_0$. This, together with \eqref{eq1}, gives us
that for any \,$l,\nu\in\N$, $z \in K_\nu$
\,and \,$n \in A(l,\nu) \setminus\{1,\dots,n_0\}$, we have
\begin{equation*}
|C_{\varphi_n}(f)(z)-P_l(z)| < \delta .
\end{equation*}
Since \,$\underline{\rm dens} \, (A(l,\nu) \setminus\{1,\dots,n_0\}) = \underline{\rm dens} \,( A(l,\nu) ) > 0$ \,and the sets
\,$B_{K_\nu}(P_l,\delta)$ 
\,form a base for the open sets of \,$H(G)$, the frequent hypercyclicity of \,$f$ \,is proved.
\end{proof}

The assumptions of Theorem \ref{existence} imply the frequent hypercylicity of the sequence of composition operators. In the next section we will establish
that, under the same conditions, the set $FHC((C_{\varphi_n})_n)$ enjoys a large algebraic size.

\vskip 4pt

In view of Proposition \ref{Prop: necessary condition} and Theorem \ref{existence}, we want to pose here the natural problem of \,{\it characterization} \,of the
frequent hypercyclicity of the sequence $(C_{\varphi_n})$ in terms of properties of $(\varphi_n)$. Note that, contrarily to the case of the iterates of
{\it one} composition operator (see \cite{bes2013} and Section 1), hypercyclicity and frequent hypercyclicity are {\it not} \,equivalent for a general sequence $(C_{\varphi_n})$. Indeed, consider a simply connected domain $G \subset \C$ and take a \,$\varphi \in H_{1-1}(G)$ without fixed points, and then define $(\varphi_n )$ as
$\varphi_{2^k} = \varphi^k$ (the compositional $k$th-iterate of $\varphi$) for $k \in \N$, and $\varphi_k(z) = z$ for $k \in \N \setminus \{2^k: \, k \in \N \}$.
It is clear that $(C_{\varphi_n})$ is hypercyclic but not frequently hypercyclic. Other --still unsolved-- questions concerning the frequent hypercyclicity of sequences of composition operators on $H(\C )$ defined by similarities were posed in \cite[Remark 4.3.2]{bernal2013}.

\section{Vector subspaces of frequently hypercyclic sequences of composition operators}

\quad As we have mentioned earlier, under the condition of being strongly frequently runaway we get a high degree of
algebraic genericity. This will be shown in the next two theorems.

\vskip 4pt

We first consider spaceability.
The following technical result 
will be needed in the proof. 
Recall that two basic sequences
\,$(x_n)_{n }, \, (y_n)_{n }$ \,in a Banach space \,$(X, \| \cdot \|)$ \,are said to be {\it equivalent} if, for every sequence
\,$(a_n)_{n }$ \,of scalars, the series \,$\sum_{n=1}^\infty a_n
x_n$ \,converges if and only if the series \,$\sum_{n=1}^\infty a_n y_n$
converges. This happens (see \cite{beauzamy1982}) if and only if there exist
two constants \,$m,M \in (0,+\infty )$ \,such that
$$
m \Big\| \sum_{j=1}^J a_j x_j \Big\| \leq \Big\| \sum_{j=1}^J a_j y_j \Big\| \leq
M \Big\| \sum_{j=1}^J a_j x_j \Big\|
$$
for all scalars $a_1, \dots ,a_J$ and all $J \in \N$. By using the
first inequality, we are easily led to the next lemma, whose proof can be found in \cite[Lemma 2.1]{bernal2006}.
By \,$L^2 (\T )$ \,we denote the Hilbert space of all Lebesgue-measurable
functions \,$f : \T \to \C$
\,with finite quadratic
norm \,$\|f\|_2 = \big( \int_0^{2\pi}|f(e^{i\theta})|\,{d\theta \over 2\pi} \big)^{1/2}$.
It is well known that the
sequence \,$(z^n)_{n }$ \,is a basic sequence in \,$L^2 ( \T )$.

\begin{lemma} \label{Lemma: basic sequences}
Assume that \,$G$ \,is a domain with \,$\overline{\D} \subset G$ \,and that
\,$(f_j)_{j } \subset H(G)$ \,is a sequence such that it is a
basic sequence in \,$L^2 (\T )$ \,that is equi\-va\-lent to
\,$(z^j)_{j }$. If \,$\Big(h_l := \sum_{j=1}^{J(l)} c_{j,l}
f_j\Big)_{l }$ \,is a sequence in \,${\rm span}\{f_j:\, {j \geq
1}\}$ \,converging in \,$H(G)$, then we have $$\dis{\sup_{l \in \N} \sum_{j=1}^{J(l)}
|c_{j,l}|^2} < +\infty.$$
\end{lemma}

\begin{theorem}\label{spaceable_thm}
Let \,$G \subset \C$ \,be a simply connected domain and \,$(\varphi_n)_n$ be a strongly frequently runaway sequence in $H_{1-1}(G)$. Then the set \,$FHC((C_{\varphi_n})_n)$ is spaceable in \,$H(G)$.
\end{theorem}

\begin{proof}
Let \,$(K_\nu)_{\nu }\subset{\mathcal M}(G)$ and \,$\{A(\nu ):\,{\nu \in\N}\}$ be, respectively, an exhaustive sequence of compact sets and a countable family of sets of positive integers given by the strongly frequent runawayness of $(\varphi_n)_n$.

\vskip 4pt

Without loss of generality, we can assume that \,$\overline{\D} \subset K_1\subset G$, because frequent hypercyclicity is stable
under translations and dilations. We are going to modify the proof of Theorem \ref{existence} appropriately to get spaceability.

\vskip 4pt

Let \,$(P_l)_l$ \,be a dense sequence of \,$H(G)$ \,and divide each set \,$A(\nu)$ \,into infinitely many disjoint subsets \,$A(\nu,l,p)$ (${l,p\in\N}$), each of them having positive lower density. Observe that condition (P3) of Definition \ref{Def: strongly freq run} implies the existence of \,$k_1 \in \N$ \,such that
\,$K_1\cap \varphi_n(K_\nu)=\varnothing$ \,for all \,$\nu\ge k_1$ \,and all \,$n\in A(\nu)$.
Similarly to the proof of Theorem \ref{existence}, it follows from conditions (P2) and (P3) that the set
$$
F := K_1 \cup \bigcup_{\nu\ge k_1} \bigcup_{n\in A(\nu)} \varphi_n(K_\nu)
$$
is a Carleman set of $G$.

\vskip 4pt

Now, for each $\mu\in\N$, let us define the function \,$g_\mu:F\to \C$ \,by
$$
g_\mu(z):=\begin{cases}z^\mu& \text{ if }z\in K_1\\
P_l(\varphi_n^{-1}(z))&\text{ if }z\in\varphi_n(K_\nu),\, \nu\ge k_1,\, n\in A(\nu,l,\mu),\, l\in\N\\
0&\text{ if }z\in\varphi_n(K_\nu),\, \nu\ge k_1,\, n\in A(\nu,l,p),\, p\ne\mu,\, l\in\N.
\end{cases}
$$
It is plain that each \,$g_\mu$ \,is continuous on \,$F$ \,and holomorphic in \,$F^\circ$.
Consequently, by Nersesjan's Theorem (Theorem \ref{Nersesjan}), there exists a function \,$f_\mu\in H(G)$ \,such that
\begin{equation}\label{spaceable}
|f_\mu (z) - g_\mu (z)| < \frac{1}{3^\mu} \cdot \min\{1,\varepsilon(z)\} \, \text{ for all } z \in F,
\end{equation}
where, again, \,$\varepsilon(z)=\chi(z,\partial_\infty G)$ \,as in the proof of Theorem \ref{existence}.

\vskip 4pt

We claim that the closed linear space generated by $(f_\mu)_\mu$ in $H(G)$, namely,
$$
M := \overline{\rm span} \, \{f_\mu :\, \mu \in \N\} ,
$$
is an infinite dimensional closed vector space consisting, except for zero, of frequently hypercyclic functions for
\,$(C_{\varphi_n})_{n }$.

\vskip 4pt

With this aim, observe first that, since \,$\overline{\D} \subset K_1$, we have \,$|f_\mu (z) - z^\mu |< \frac{1}{3^\mu}$
\,for all \,$z \in \T$.
Let \,$(e_\mu^*)_{\mu}$ \,be the sequence of coefficient
functionals corresponding to the basic sequence \,$(z^\mu)_{\mu }$ \,in \,$L^2 ( \T )$.
Since \,$\|e_\mu^*\|_2 = 1$ $(\mu \in \N )$, one obtains
\begin{equation} \label{eq perturbation}
\sum_{\mu =1}^\infty \|e_\mu^*\|_2 \cdot \|f_\mu - z^\mu\|_2 <
\sum_{\mu =1}^\infty {1 \over 3^\mu} = {1 \over 2} < 1.
\end{equation}
From (\ref{eq perturbation}) and the basis perturbation
theorem \cite[p.~50]{diestel1984} it follows that
$(f_\mu)_{\mu }$ \,is also a basic sequence in \,$L^2(\T )$ \,that
is equivalent to \,$(z^\mu)_{\mu}$. In particular, the functions
\,$f_\mu$ ($\mu \in \N$) \,are linearly independent
and \,$M$ \,is an infinite dimensional closed vector space.

\vskip 4pt

Now, fix any \,$f \in M \setminus\{0\}$ \,and let
\,$f = \sum_{\mu\in\N} \alpha_\mu f_\mu$ \,be its representation in \,$L^2(\T)$. As \,$f \ne 0$ \,there is some nonzero coefficient \,$\alpha_\mu$,
which without loss of generality may supposed to be $\alpha_1$
(indeed, if $\alpha_m$, and not $\alpha_1$, is the first nonzero coefficient, the reasonings below would involve sums
$\sum_{\mu = m}^{N_j}$, $\sum_{\mu = m+1}^{N_j}$, $\sum_{\mu = m}^{\infty}$ and the inequality $|f_m(\varphi_n(z))-P_l(z)| < \ve (\varphi_n(z))$, which produce
the same effects as if $m=1$).
Furthermore, 
by the invariance under scalar multiplication of frequent hypercyclicity,
we can assume that \,$\alpha_1 = 1$.
Thanks to the definition of \,$M$, there is a sequence \,$\Big( h_j := \sum_{\mu=1}^{N_j} \alpha_{j,\mu} f_\mu \Big)_{j }$
\,converging to \,$f$ \,in \,$H(G)$. Hence
$$h_j \longrightarrow f \hbox{ \ in \ }  L^2(\T ),$$
because the topology in \,$H(G)$ \,is finer than the one in \,$L^2(\T )$.
By the continuity of each projection, we obtain that \,$\alpha_{j,1}\to 1$  \,($j \to\infty$), and without loss
of generality we may assume that \,$\alpha_{j,1} = 1$ \,for all \,$j \in \N$ \,(if not, just take $H_j := h_j+(1-\alpha_{1,j})f_1$).
Finally, by Lemma \ref{Lemma: basic sequences}, there is a constant \,$H>0$ \,such that
\begin{equation} \label{Inequality}
\sum_{\mu=1}^{N_j}|\alpha_{j,\mu}|^2\le H \hbox{ \ for all \ } j \in \N .
\end{equation}

Next, fix \,$l, \nu \in\N$ \,with \,$\nu\ge k_1$ \,and \,$n\in A(\nu,l,1)$.
On the one hand, the set \,$\varphi_n (K_\nu)$ \,is compact. Then,
given \,$\delta > 0$, there exists \,$j = j(n) \in \N$ \,such that
\begin{equation}\label{sp1}
|f(\varphi_n(z)) - h_j(\varphi_n(z))|< {\delta \over 2} \hbox{ \ \,for all \ } z \in K_\nu .
\end{equation}
But on the other hand, applying \eqref{spaceable} and \eqref{Inequality} together with the Cauchy--Schwarz inequality guarantees that, for all $z \in K_\nu$,
\begin{eqnarray} \label{sp2}
|h_j(\varphi_n(z))-P_l(z)|&\le&\dis |f_1(\varphi_n(z))-P_l(z)|+\sum_{\mu=2}^{N_j}|\alpha_{j,\mu}|\, |f_\mu(\varphi_n(z))|\nonumber\\
&\dis \le&\dis \varepsilon(\varphi_n(z))+\left(\sum_{\mu=2}^{N_j}|\alpha_{j,\mu}|^2\right)^{1/2} \left(\sum_{\mu=2}^{N_j}|f_\mu(\varphi_n(z))|^2\right)^{1/2}\nonumber\\
&\dis < &\dis \varepsilon(\varphi_n(z)) + \sqrt{H} \, \sum_{\mu=2}^\infty\frac{1}{3^\mu} \cdot \varepsilon(\varphi_n(z))\nonumber\\
&\dis < &\dis (1+\sqrt{H}) \cdot \varepsilon(\varphi_n(z)).
\end{eqnarray}
Given \,$\nu \in \N$, there is \,$n_0 = n_0(\nu ) \in \N$ \,such that \,$\sup_{z \in K_\nu} \ve (\varphi_n (z)) < \delta / (2 + 2\sqrt{H})$ \,for \,$n \ge n_0$.
Finally, from \eqref{sp1}, \eqref{sp2} and the triangle inequality (which, incidentally, neutralizes the effect of the dependence $j=j(n)$), an argument similar to that of the end of the proof of
Theorem \ref{existence} yields frequent $(C_{\varphi_n})_n$-hypercyclicity for \,$f$. This concludes the proof.
\end{proof}

A slight modification of the arguments in the proof of the last result allows us
to prove the maximal dense lineability of \,$FHC((C_{\varphi_n})_n)$ as well.

\begin{theorem}\label{dense-lineable-thm}
Let \,$G \subset \C$ \,be a simply connected domain and \,$(\varphi_n)_n$ be a strongly frequently runaway sequence in $H_{1-1}(G)$. Then the set \,$FHC((C_{\varphi_n})_n)$ is maximal dense-lineable.
\end{theorem}

\begin{proof}
Again, let $(K_\nu)_{\nu }$ and $\{A(\nu ):\,{\nu \in\N}\}$ be an exhaustive sequence of compact sets and a countable family of subsets of $\N$ given by the strongly frequent runawayness of $(\varphi_n)_n$. Without loss of generality we may assume that \,$\overline{\D}\subset K_1$.

\vskip 4pt

Let $(P_l)_l$ be a dense sequence in \,$H(G)$ \,and divide each set \,$A(\nu)$ \,into infinitely many mutually
disjoint subsets \,$A(\nu,l,p)$ $(l,p \in \N )$, each of them with positive lower density. For \,$\mu\in\N$ fixed, by (P3) there
exists \,$k_\mu \in \N$ \,such that for all \,$\nu \ge k_\mu$ \,and all \,$n\in A(\nu)$, we have $K_\mu \cap \varphi_n(K_\nu)=\varnothing$.
Note that the sequence \,$(k_\mu)_\mu \subset \N$ \,can be chosen to be strictly increasing.
Similarly to the proofs of Theorems \ref{existence} and \ref{spaceable_thm},
one derives from (P2) and (P3) that every set
$$
F_\mu := K_\mu \cup \bigcup_{\nu\ge k_\mu} \bigcup_{n\in A(\nu)} \varphi_n (K_\nu)
$$
is a Carleman set.


\vskip 4pt

For each \,$\mu\in\N$, let us define the function \,$g_\mu : F_{\mu+1} \to \C$ as follows:
$$
g_\mu(z):=\begin{cases}
P_{\mu}(z)&\text{ if }z\in K_{\mu+1}\\
P_l(\varphi_n^{-1}(z))&\text{ if }z\in \varphi_n(K_\nu),\,\nu\ge k_{\mu+1},\, n\in A(\nu,l,\mu+1),\, l\in\N\\
0&\text{ if }z\in \varphi_n(K_\nu),\,\nu\ge k_{\mu+1},\, n\in A(\nu,l,p),\,p\ne\mu+1,\,l\in\N.
\end{cases}
$$
Observe that $g_\mu(z)=0$ for all $z\in\varphi_n(K_\nu)$ where $n\in A(\nu,l,1)$, $\nu\ge k_{\mu+1}$ and $l\in\N$.
We are going to use these compact sets later.

\vskip 4pt

As in the proof of Theorem \ref{spaceable_thm}, each \,$g_\mu$ \,is continuous on \,$F_{\mu + 1}$ \,and holomorphic on $F_{\mu + 1}^\circ$.
Hence, for each \,$\mu \in \N$, Theorem \ref{Nersesjan} guarantees the existence of a function \,$f_\mu \in H(G)$ \,such that \begin{equation}\label{dense-lineable}
|f_\mu(z) - g_\mu(z)| < \frac{1}{\mu} \cdot \min \{1, \varepsilon(z) \} \text{ \ for all \ } z \in F_{\mu +1},
\end{equation}
where $\varepsilon(z)=\chi (z,\partial_\infty G)$. From \eqref{dense-lineable}, we get, in particular, that \,$|f_\mu(z)-P_\mu(z)| < \frac{1}{\mu}$ \,for all \,$z\in K_{\mu+1}$. So
$$
\lim_{\mu\to\infty} \|f_\mu-P_\mu\|_{K_{\mu+1}} = 0
$$
and, since the sequence \,$(P_\mu)_\mu$ \,is dense in \,$H(G)$ \,and \,$(K_\mu)_\mu$ \,is an exhaustive sequence of compact sets,
we derive that \,$(f_\mu)_\mu$ \,is dense in \,$H(G)$.
Now, we define the set
$$
M_d := {\rm span} \{f_\mu:\, \mu \in \N\} ,
$$
the linear span of the latter sequence. It is plain that \,$M_d$ \,is a dense vector subspace of \,$H(G)$.
We claim that \,$M_d \setminus\{0\} \subset FHC((C_{\varphi_n})_n)$.

\vskip 4pt

To this end, fix \,$H = \sum_{\mu=1}^N\lambda_\mu f_\mu \in M_d \setminus \{0\}$.
Without loss of generality, we may assume that \,$\lambda_N=1$ \,(because of the stability of frequent hypercyclicity under scaling).
Fix \,$l,\nu \in \N$ \,with \,$\nu \ge k_{N+1}$, $n\in A(\nu,l,N+1)$ \,and \,$z\in K_\nu$, so that \,$\varphi_n(z) \in \varphi_n(K_\nu)$.
It follows from \eqref{dense-lineable} that
\begin{eqnarray*}
|H(\varphi_n(z))-P_l(z)|&\le &\dis |f_N(\varphi_n(z))-P_l(z)|+\sum_{\mu=1}^{N-1}|\lambda_{\mu}|\, |f_\mu(\varphi_n(z))|\nonumber\\
&\dis < &\dis \alpha \cdot \varepsilon(\varphi_n(z)), 
\end{eqnarray*}
where $\alpha :=  1 + \sum_{\mu=1}^{N-1}|\lambda_\mu|$.
Continuing as in the final part of the proof of Theorem \ref{existence}
-- but in a somewhat easier way -- we get the frequent hypercyclicity of the function \,$H$, so proving the claim.

\vskip 4pt

Observe that, from \eqref{dense-lineable} and the definition of $g_\mu$, we also obtain the next ine\-qua\-li\-ty, that will be used later:
\begin{equation}\label{a bound for H(phin)}
|H(\varphi_n (z))| < \al \cdot \varepsilon (\varphi_n (z)) \text{ \ for all \ } z \in K_\nu, \, n \in A(\nu ,l,1), \, \nu \ge k_{l+1}, \, l \in \N .
\end{equation}

Now, we go back to the sets \,$A(\nu,l,1)$. For fixed \,$\nu,l\in\N$, let us divide \,$A(\nu,l,1)$ \,into infinitely many sequences
\,${A(\nu,l,1,\mu)}$ ($\mu\in\N$), each of them with positive lower density.
Let \,$F := F_1 = K_1 \, \cup \, \bigcup_{\nu \ge k_1} \bigcup_{n \in A(\nu)} \varphi_n (K_\nu)$
\,and proceed as in the proof of Theorem \ref{spaceable_thm}, but defining \,$h_\mu : F \to \C$ $(\mu \in \N )$ \,in the next way:
$$
h_\mu (z) :=
\begin{cases}z^\mu& \text{ if }z\in K_1\\
P_l(\varphi_n^{-1}(z))&\text{ if }z\in\varphi_n(K_\nu),\, \nu\ge k_1,\, n\in A(\nu,l,1,\mu ),\, l \in \N\\
0&\text{ elsewhere.}
\end{cases}
$$
Observe that, for \,$p \ge 2$ \,and \,$l \in \N$, we have \,$h_\mu(z)=0$ \,if \,$z\in\varphi_n(K_\nu)$ \,with $\nu \ge k_1$ \,and \,$n \in A(\nu,l,p)$.

\vskip 4pt

Next, by applying Theorem \ref{Nersesjan} as in the proof of
Theorem \ref{spaceable_thm}, we get functions \,$\Phi_\mu \in H(G)$ $(\mu \in \N )$ \,satisfying
$$
|\Phi_\mu (z) - h_\mu (z)| < \frac{1}{3^\mu} \cdot \min\{1,\ve(z)\} \, \text{ for all } \, z \in F.
$$
Proceeding as in the last mentioned theorem -- except for the fact that only the sequences $A(\nu,l,1,\mu)$ $(\mu \in \N )$ are handled instead
of all sequences $A(\nu , l, \mu )$ -- it is obtained that the functions \,$\Phi_\mu$ $(\mu \in \N )$ \,are
linearly independent and that \,$M_s \setminus \ \{0\} \subset FHC((C_{\varphi_n})_n)$, where
\,$M_s := \overline{{\rm span}}\{\Phi_\mu : \,\mu \in \N \}$. Since \,$M_s$ \,is a closed infinite dimensional
vector subspace of the separable F-space \,$H(G)$, an application of Baire's category theorem
yields \,${\rm dim} \, (M_s) =\mathfrak{c} = {\rm dim} \, (H(G))$, where \,$\mathfrak{c}$ \,denotes the cardinality of continuum.
Specifically, given \,$\Phi \in M_s \setminus \{0\}$, $l,\nu \in \N$ \,with \,$\nu$ \,large enough and \,$\delta > 0$, there is
some \,$\mu_0 \in \N$ \,such that
\begin{equation} \label{eqn Ms}
|\Phi (\varphi_n(z)) - P_l(z)| < \delta /2
\end{equation}
for all \,$n \in A(\nu ,l,1,\mu_0)$ \,large enough and all \,$z \in K_\nu$.

\vskip 4pt

Finally, consider the vector subspace of \,$H(G)$ \,given by
$$
M_{\rm max} := M_d + M_s.
$$
Note that \,$M_{\rm max}$ \,is dense
(because it contains $M_d$) and has dimension $\mathfrak{c} = {\rm dim} \, (H(G))$ (because it contains $M_s$, and ${\rm dim} (M_s) = \mathfrak{c}$).
Fix a function \,$f \in M_{\rm max} \setminus \{0\}$. Two cases are possible. If \,$f \in M_d$, then \,$f \in FHC((C_{\varphi_n})_n)$.
If \,$f \not\in M_d$, then there is a function \,$\Phi \in M_s \setminus \{0\}$ (as in the preceding paragraph) and
another function \,$H \in M_d$
\,such that \,$f = \Phi + H$. Therefore \,$(C_{\varphi_n} f)(z)  = \Phi (\varphi_n(z)) + H (\varphi_n (z))$.
Given \,$l,\nu \in \N$ \,with \,$\nu$ \,large enough,
we had obtained (see \eqref{a bound for H(phin)}) the existence of \,$\al > 0$ \,such that \,$| H (\varphi_n (z)) | < \al \cdot \chi (\varphi_n (z),\partial_\infty G)$
for all \,$z \in K_\nu$ \,and all \,$n \in A(\nu ,l,1)$ (so for all \,$n \in A(\nu ,l,1,\mu_0)$).
Since \,$\sup_{z \in K_\nu} \chi (\varphi_n (z),\partial_\infty G) \to 0$ \,as \,$n \to \infty$, we get
for \,$n \in A(\nu ,l,1,\mu_0)$ \,large enough that
\,$| H (\varphi_n (z))| < \delta /2$ \,for all \,$z \in K_\nu$.
Thanks to (\ref{eqn Ms}) and the triangle inequality, we get
$|(C_{\varphi_n} f)(z) - P_l(z)| < \delta$ \,for the same \,$n,z$.
Then the proof of the fact \,$f \in FHC ((C_{\varphi_n})_n)$ \,is concluded as soon as we take into account the density of
\,$(P_l)_{l}$ \,in \,$H(G)$ \,and the property \,$\underline{\rm dens} \, ( A(\nu ,l,1,\mu_0) ) > 0$.
\end{proof}

\begin{remarks}
{\rm 1. Let \,$G \subset \C$ \,be \,{\it any} \,domain of \,$\C$ \,and \,$(\varphi_n)_n \subset H(G,G)$, and let us consider the family
$FHC_{\mathcal M} ((C_{\varphi_n})_n) := \{f \in H(G) : \, \underline{\rm dens} (\{n \in \N : \, f \circ \varphi_n \in B_K(g,\delta)\}) > 0$
for each $K \in \mathcal M (G)$, each $\delta > 0$ and each $g \in H(G)\}$.
Under the same assumptions as in Theorems \ref{existence}-\ref{spaceable_thm}-\ref{dense-lineable-thm} (the strongly frequent runawayness of $(\varphi_n)_n$ in $H_{1-1}(G)$), and with the same proofs, it is obtained
that the set \,$FHC_{\mathcal M} ((C_{\varphi_n})_n)$ is spaceable and maximal dense-lineable in \,$H(G)$.
Note that, if \,$G$ \,is multiply connected, then $FHC_{\mathcal M} ((C_{\varphi_n})_n)$ is bigger than $FHC ((C_{\varphi_n})_n)$.
In fact, if \,$G$ \,is a finitely connected domain that is not simply connected
and \,$(\varphi_n )_n \subset H_{1-1}(G)$, then the set \,$HC ((C_{\varphi_n})_n)$ (so the set \,$FHC ((C_{\varphi_n})_n)$ \,as well) is empty:
see \cite[Theorem 3.15]{grossemortini2009}.

\vskip 2pt

\noindent 2. According to B\`es \cite{bes2013} (see Section 1), if $G$ is simply connected and $\varphi \in H(G,G)$, then $C_\varphi$ is hypercyclic if and only if
it is frequently hypercyclic, and if and only if $FHC(C_\varphi )$ is spaceable. Since we do not know whether the property of being strongly frequently runaway for $(\varphi_n)$  characterizes the frequent hypercyclicity of $(C_{\varphi_n})$, a natural \,{\it question} \,arises for sequences of composition operators: If $\{\varphi_n\}_{n \ge 1} \subset H_{1-1}(G)$ is such that $(C_{\varphi_n})$ is frequently hypercyclic, is $FHC((C_{\varphi_n}))$ spaceable (and even maximal dense-lineable)?}
\end{remarks}

\section{Examples and final remarks}

1. The first example is motivated by one in \cite{bernal2013}.
Consider the slit complex plane \,$G := \C \setminus (-\infty ,0]$.
Let $\alpha , \beta \in \R$ with $\beta > 0$ and $\beta \ge 1 +\al$.
For \,$N \in \N$, $z^{1/N}$ \,will represent the principal branch of the
$N$th root of \,$z$ \,in \,$G$. Fix $N \in \N$. Then the mappings
$$\varphi_n (z) := n^\al z^{1/N} + n^\beta \ \ (n \in \N )$$
belong to \,$H_{1-1}(G) \setminus {\rm Aut} \, (G)$.
Let \,$C > 0$ \,be a constant to be specified later, and consider the
numbers $R_\nu = (C \, \nu^{\beta - \al})^N$ and the compact sets
$$
K_\nu = \Big\{z = re^{i \theta} : \, \min \big\{{1 \over R_\nu} , 1 \big\} \le r \le R_\nu , \, |\theta | \le \pi \big(1 - {1 \over \nu} \big) \Big\}
\ (\nu \in \N ).
$$
Note that $K_\nu \subset \ovl{D}(0,R_\nu )$, the closed disc with center \,$0$ \,and radius \,$R_\nu$.
Hence \,$\varphi_n (K_\nu ) \subset \ovl{D}(n^\beta , n^\al R_\nu^{1/N} )$.
Plainly, all $K_\nu$'s lack holes and form an exhaustive sequence of compact subsets of $\C \setminus (-\infty ,0]$.
By \cite[Lemma 2.2]{bonillagrosse2007}, there exist
pairwise disjoint subsets \,$A(l,\nu )$ $(l,\nu \ge 1)$ \,of \,$\N$ with \,$\underline{\rm dens} (A(l,\nu )) > 0$
\,such that, for any \,$n \in A(l,\nu )$ \,and \,$m \in A(k,\mu )$, we have \,$n \ge \nu$ \,and \,$|n-m| \ge \nu + \mu$ \,if \,$n \ne m$.
Define $A(\nu ) := \bigcup_{l \in \N} A(l,\nu )$. Of course, $\underline{\rm dens} (A (\nu )) > 0$ for every $\nu \in \N$, and the family
\,$\{A(\nu ): \, \nu \in \N\}$ \,is pairwise disjoint.
Pick two distinct members \,$A,B \in \mathcal K := \{\varphi_n (K_\nu ) : \, n \in A(\nu ), \, \nu \in \N\}$.
Then there are distinct $m,n \in \N$ ($m > n$, say) as well as sets $A(l,\nu ), \, A(k,\mu )$ such that
\,$n \in A(l,\nu ), \,m \in A(k,\mu ), \, A = \varphi_m (K_\mu )$ \,and \,$B = \varphi_n (K_\nu )$.
It is easy to see that \,$\sigma := \inf_{t > 1} {t^\beta - 1 \over t^\al (t-1)^{\beta - \al}} > 0$. Take \,$C := \min \{1/2,\sigma /4 \}$.
Since $m > n$, we get \,${(m/n)^\beta - 1 \over (m/n)^\al ((m/n)-1)^{\beta - \al}} \ge 4 \, C$, that is,
${m^\beta - n^\beta \ge  4 \, C \cdot m^\al (m - n)^{\beta - \al}}$. The distance between the centers of
the discs \,$D_1 := \ovl{D}(m^\beta , m^\al R_\mu^{1/N} )$ \,and \,$D_2 := \ovl{D}(n^\beta , n^\al R_\nu^{1/N} )$ \,is \,$m^\beta - n^\beta$.
But observe that
\begin{equation*}
\begin{split}
m^\beta - n^\beta &>  4 \, C \cdot {m^\al + n^\al \over 2} \cdot (m - n)^{\beta - \al}
                                                         \ge  4 \, C \cdot {m^\al + n^\al \over 2} \cdot (\mu + \nu )^{\beta - \al} \\
                  &\ge 2 \, C \cdot {m^\al + n^\al \over 2} \cdot (\mu^{\beta - \al} + \nu^{\beta - \al})
                                                         \ge m^\al C \, \mu^{\beta - \al} + n^\al C \, \nu^{\beta - \al} \\
                  &= {\rm radius} \, (D_1) + {\rm radius} \, (D_2).
\end{split}
\end{equation*}
Consequently, $D_1 \cap D_2 = \varnothing$ \,and so \,$A \cap B = \varnothing$. Then \,$\mathcal K$ \,is pairwise disjoint. Finally, fix
a compact set \,$K \subset \C \setminus (-\infty ,0]$. There is \,$R > 0$ \,such that \,$K \subset \ovl{D}(0,R)$. If $L \in \mathcal K$, there exists
$(l,\nu ) \in \N^2$ and $n \in A(l,\nu )$ (so $n \ge \nu$) with $L = \varphi_n (K_\nu ) \subset \ovl{D}(n^\beta , n^\al R_\nu^{1/N} )$. It follows that each
$w \in L$ satisfies $|w - n^\beta| \le n^\al R_\nu^{1/N}$, hence (recall that $C \le 1/2$) we get
\,$|w| \ge n^\beta - n^\al R_\nu^{1/N} = n^\beta - n^\al C \, \nu^{\beta - \al}
\ge n^\beta - n^\al C \, n^{\beta - \al} \ge (1/2) n^\beta > R$ \,if \,$n$ \,is large enough. Therefore $K \cap \varphi_n (K_\nu ) = \varnothing$, so
\,$K \cap L = \varnothing$ \,except for finitely many \,$L \in \mathcal K$, because no \,$n$ \,may belong to two different $A(l,\nu )$'s.
According with Theorems \ref{existence}-\ref{spaceable_thm}-\ref{dense-lineable-thm}, $(C_{\varphi_n})_n$ is
frequently hypercyclic and the set $FHC((C_{\varphi_n})_n)$ \,is spaceable and
maximal dense-lineable in \,$H(\C \setminus (-\infty ,0])$.

\vskip 5pt

2. If \,$G$ \,and \,$\Omega$ \,are simply connected domains different from \,$\C$, then the Riemann conformal representation theorem (see, e.g., \cite{ahlfors1979}) provides an isomorphism (that is, a bijective biholomorphic mapping) \,$f : G \to \Omega$.
Therefore, if \,$\varphi \in H(G,G)$, we have \,$f \circ \varphi \circ f^{-1} \in H(\Omega , \Omega )$ \,and the
mapping \,$h \in H(\Omega ) \mapsto h \circ f \in H(G)$ \,is a linear homeomorphism.
Hence eve\-ry example of a sequence \,$(\varphi_n )_n \subset H(G,G)$ \,satisfying that \,$FHC ((C_{\varphi_n})_n)$ \,is nonempty
(spaceable, maximal dense-lineable, resp.) provides us with an example of a sequence \,$(\Phi_n)_n \subset H(\Omega,\Omega)$ \,such that
\,$FHC ((C_{\Phi_n})_n)$ \,is nonempty (spaceable, maximal dense-lineable, resp.). Indeed, take \,$\Phi_n := f \circ \varphi_n \circ f^{-1}$ $(n \in \N )$.
For instance, we have in accordance with Example 1 that \,$FHC((C_{\varphi_n})_n)$ \,is spaceable and maximal
dense-lineable in \,$H(\C \setminus (-\infty ,0])$, where \,$\varphi_n (z) = n + z^{1/N}$. Moreover, the mapping
\,$f : z \in \C \setminus (-\infty ,0] \mapsto {z^{1/2} - 1 \over z^{1/2} + 1} \in \D$ \,is an isomorphism
with inverse \,$f^{-1}(z) = \left( {1+z \over 1-z} \right)^2$. Consequently, the set
$FHC((C_{\Phi_n})_n)$ is spaceable and maximal dense-lineable in $H(\D )$, where
\,$\Phi_n (z) = {\big(n + (({1+z \over 1-z})^2)^{1/N} \big)^{1/2} - 1 \over \big(n + (({1+z \over 1-z})^2)^{1/N} \big)^{1/2} + 1}$.

\vskip 5pt

3. Let $a > 0$ and $\gamma \ge 1$. It is easy to check that every linear fractional function
\,$\Phi_n(z) := 1 + {2(z-1) \over 2 - ian^\gamma (z-1)}$ $(n \in \N )$ \,is an automorphism of $\D$.
Moreover, every $\Phi_n$ is parabolic (see \cite[pp.~6--7]{shapiro1993}), that is, it has a unique fixed point, located at $\T$ (namely, at $1$).
With arguments similar -- but easier -- to those given in Example 1, we can check that if $\gamma \ge 1$,
$\Pi_+ := \{z \in \C :$ Re$\, z > 0\}$ is the open right half-plane and $\varphi_n (z) := z + i a n^\gamma$ \,for each $n \in \N$
(note that $(\varphi_n )_n \subset {\rm Aut} \, (\Pi_+ )$), then $FHC((C_{\varphi_n})_n)$ is spaceable and maximal dense-lineable in $H(\Pi_+ )$.
Since \,$f(z) := {1 + z \over 1 - z}$ \,is an isomorphism between \,$\D$ \,and \,$\Pi_+$ \,with inverse \,${z - 1 \over z + 1}$, and
since \,$\Phi_n = f^{-1} \circ \varphi_n \circ f$, we obtain as in Example 2 the spaceability as well as the maximal dense-lineability for
\,$FHC((C_{\Phi_n})_n)$ \,in \,$H(\D )$. The case \,$\gamma = 1$ \,gives the same properties (spaceability was already known, see Section 1) for
the set \,$FHC(C_\varphi )$, where \,$\varphi$ \,is the
parabolic automorphism of \,$\D$ \,given by \,$\varphi (z) = 1 + {2(z-1) \over 2 - i a (z-1)}$.

\vskip 5pt

4. Let \,$G = \C$. In this case \,$H_{1-1}(G) = {\rm Aut} \, (\C ) = \{$similarities $z \mapsto az+b : \, a,b \in \C , \, a \ne 0\}$.
Examples of sequences of similarities $(\varphi_n (z) = a_n z + b_n)_{n }$ \,for which $(C_{\varphi_n})_n$ is frequently hypercyclic were
provided in \cite[Theorem 4.1]{bernal2013} (see Section 1). In its proof, it is shown in fact that these sequences $(\varphi_n)$ are strongly
frequently runaway. Then, according to Theorems \ref{existence}-\ref{spaceable_thm}-\ref{dense-lineable-thm}, we obtain that for each of such
sequences the set $FHC((C_{\varphi_n})_n)$ is not only nonempty but also spaceable and maximal dense-lineable in $H(\C )$.

\vskip 5pt

5. In this final and rather general example --from which the parabolic automorphism \,$\varphi$ \,defined in the last line of Example 3 above is a particular instance-- we will not use Theorems \ref{existence}-\ref{spaceable_thm}-\ref{dense-lineable-thm}.
We say that a domain \,$G \subset \C$ \,is a \textit{Jordan domain} whenever $\partial_\infty G$ \,is a homeomorphic
image of \,$\T$ (hence $G$ need not be bounded; for instance, an open half-plane is a Jordan domain).
We exhibit the mentioned collection of examples in the following proposition.

\begin{proposition}
The following assertions hold:
\begin{itemize}
\item[(a)] Assume that \,$G$ \,is a Jordan domain in \,$\C$ \,and that \,$\varphi \in H_{1-1}(G)$ \,satisfies that
there is \,$\xi \in \partial_\infty G$ \,such that
\,$\varphi^{n} \to \xi$ \,as \,$n \to \infty$ \,uniformly on compacta in \,$G$.
Then \,$FHC(C_\varphi )$ \,is maximal dense-lineable in \,$H(G)$.
\item[(b)] If \,$\varphi$ is a non-elliptic automorphism of \,$\D$, then \,$FHC(C_\varphi )$ is maximal dense-lineable in $H(\D )$.
\end{itemize}
\end{proposition}

\begin{proof}
(a) As a particular instance of a result by B\`es \cite{bes2013}
mentioned in Section 1, the set \,$FHC(C_\varphi )$ \,is spaceable if it is assumed that \,$\varphi \in H_{1-1}(G)$ \,and
\,there is \,$\xi \in \partial_\infty G$ \,such that
\,$\varphi^{n} \to \xi$ \,as \,$n \to \infty$ \,uniformly on compacta in \,$G$:
indeed, such a \,$\varphi$ \,cannot have fixed points in \,$G$.
According to Osgood--Carath\'eodory's theorem (see \cite{henrici1986}), there is an isomorphism \,$f : G \to \D$ \,that is extendable
to a homeomorphism \,$\ovl{G}^\infty \to \ovl{\D}$, where \,$\ovl{G}^\infty$ \,is the closure in \,$\C_\infty$ \,of \,$G$. Since the polynomials form
a dense set in $H(\D )$, the set \,$\mathcal D := \{P \circ f : \, P$ polynomial$\}$ \,is a dense vector subspace of \,$H(G)$.
Now, the sequence $((C_\varphi)^{n} Q )_{n }$ converges in $H(G)$ for every $Q \in \mathcal D$, because if
\,$Q = P \circ f$ \,then
$$
(C_\varphi)^{n} Q= C_{\varphi^{n}} Q = P \circ f \circ \varphi^{n} \longrightarrow P(f(\xi )) \quad (n \to \infty )
$$
uniformly on compacta in $G$.
Since $FHC(C_\varphi ) = FHC((C_{\varphi^{n}} )_n)$ is spaceable and $H(G)$ is a separable F-space, from \cite[Theorem 4.10(a)]{bernalordonez2014}
we can conclude that $FHC(C_\varphi )$ is maximal dense-lineable in $H(G)$.

\vskip 2pt

\noindent (b) This is a special instance of (a) because for every non-elliptic $\varphi \in {\rm Aut} (\D )$ there is a point $\xi \in \T$ (the unique fixed point if $\varphi$ is parabolic, and the attractive fixed boundary point if $\varphi$ is hyperbolic) such that \,$\varphi^{n} \to \xi$ \,as \,$n \to \infty$ \,uniformly on compacta in \,$\D$ (see \cite{shapiro1993}).
\end{proof}


6. In the proof of Theorems \ref{existence}-\ref{spaceable_thm}-\ref{dense-lineable-thm}, we have not used any property of the lower density except that if \,$A$ \,is a subset of \,$\N$ \,with \,$\underline{\rm dens} (A) > 0$, then \,$A$ \,can be
divided into an infinity of disjoint sets \,$A_n$ \,with \,$\underline{\rm dens} (A_n) > 0$. Therefore, our results in Theorems \ref{existence}-\ref{spaceable_thm}-\ref{dense-lineable-thm} will also be verified for the upper-frequent hypercyclicity
--introduced by Shkarin in \cite{shkarin2009}-- by replacing in the definition of ``strongly frequently runaway'' (Definition \ref{Def: strongly freq run}) the condition (P1) by $\overline{\rm dens} (A(\nu)) > 0$ for all $\nu\in\mathbb{N}$. More generally, the notion of $\mathcal A$-hypercyclicity has been recently introduced in a paper of
B\`es-Menet-Peris-Puig \cite{besmenetperispuig2016}, where an operator $T : X \to X$ is said to be $\mathcal A$-hypercyclic if there exists $x \in X$ such
that, for every nonempty open set $U \subset X$, we have $N(x,U) \in \mathcal A$. Here $\mathcal A$ is a nonempty collection of nonempty subsets of \,$\N$ satisfying appropriate axioms and $N(x,U) := \{n \in \N : \, T^n x \in U\}$. The extension of these notions to a sequence \,$T_n : X \to X$ $(n \ge 1)$ \,is immediate.
Therefore, if for every \,$A \in \mathcal A$ \,there is an infinity of mutually disjoint sets $A_n \subset A$ such that $A_n \in \mathcal A$ for all $n\in\mathbb{N}$, then our Theorems
\ref{existence}-\ref{spaceable_thm}-\ref{dense-lineable-thm} imply corresponding results for $\mathcal A$-hypercyclicity by replacing (P1) by
$A(\nu ) \in \mathcal A$ for all $\nu\in\mathbb{N}$. In particular, mere hypercyclicity is equivalent to $\mathcal A$-hypercyclicity for
$\mathcal A = \{$infinite subsets of $\N \}$ and, in this case, the adapted (P1), (P2) and (P3) are equivalent to the ``classical'' runaway condition
introduced in \cite{bernalmontes1995}.

\vskip 8pt

\noindent {\bf Acknowledgements.} The first, second and fourth authors have been partially supported by Plan
Andaluz de Investigaci\'on de la Junta de Andaluc\'{\i}a FQM-127
Grant P08-FQM-03543 and by MEC Grant MTM2015-65242-C2-1-P. The third author has been supported by DFG-Forschungsstipendium JU 3067/1-1.


\end{document}